\documentclass[10pt]{article}
\usepackage{amscd, amsthm, color}
\usepackage{mathrsfs}
\input epsf
\usepackage{amssymb}
\usepackage{amsmath}
\usepackage{amsfonts}
\input amssym.def
\newtheorem{theorem}{Theorem}[section]

\newtheorem{lemma}[theorem]{Lemma}
\newtheorem{prop}[theorem]{Proposition}

\newtheorem{re}[theorem]{Remark}
\newtheorem{no}[theorem]{Notation}
\newtheorem{definition}[theorem]{Definition}
\theoremstyle{definition}

\definecolor{wco}{rgb}{0.5,0.2,0.3}

\numberwithin{equation}{section}

\begin{document}
\title{Widths of embeddings
in weighted function spaces
%\thanks{The authors are partly supported by NNSF... of China.
% The second author is partly supported by ..}
}
\author{Shun Zhang $^{a,\, b}$\ \ and\ \ Gensun Fang
$^{a,\,}$\footnote{Corresponding author.
\newline\indent\ \, E-mail addresses: fanggs@bnu.edu.cn (G. Fang), shzhang27@163.com (S. Zhang).}
\\ {\small $^{a}$ School of Mathematical Sciences, Beijing Normal University,
Beijing 100875, China}
\\ {\small $^{b}$ School of Computer Science and Technology, Anhui
University,
 Hefei 230039, Anhui, China}}
 \maketitle
\begin{abstract}
 We study the asymptotic behaviour of the
approximation, Gelfand and Kolmogorov numbers of the compact
embeddings of weighted function spaces of Besov and Triebel-Lizorkin
type in the case where the weights belong to a large class. We
obtain the exact estimates in almost all nonlimiting situations
where the quasi-Banach setting is included. At the end we present
complete results on related widths for polynomial weights with small
perturbations, in particular the sharp estimates in the case $\alpha
= d (\frac 1{p_2}-\frac 1{p_1})>0$ therein.
\end{abstract}

\section{Introduction}
This is a direct continuation of \cite{ZF10, ZF11} on n-widths of
compact embeddings $B_{p_1,q_1}^{s_1}(\mathbb{R}^d,
w_1)\hookrightarrow B_{p_2,q_2}^{s_2}(\mathbb{R}^d, w_2)$ of
weighted Besov spaces. In these articles we considered the case
where the ratio of the weights $w(x):= w_1(x)/w_2(x)$ is of
polynomial type, and determined the asymptotic degree of the Gelfand
and Kolmogorov numbers of the corresponding embeddings. Also,
Skrzypczak \cite{Sk05} investigated the approximation numbers of the
embeddings.

In the present paper we turn our attention to more general weight
cases, and investigate the asymptotic behaviour of the
approximation, Gelfand and Kolmogorov numbers of the corresponding
embeddings. This problem has been suggested recently by K$\ddot{\rm
u}$hn et al. \cite{KLSS06} for compactness and asymptotic estimates
of the entropy numbers.

Let us make an agreement throughout this paper,
\begin{equation}\label{emB_con} -\infty<s_2<s_1<\infty,\
0< p_1, p_2, q_1, q_2\leq\infty\ \ {\rm and}\ \
\delta=s_1-s_2-d(\frac 1{p_1}-\frac 1{p_2})>0
\end{equation}
if no further restrictions are stated.

Our main results are Theorems \ref{an}, \ref{kn} and \ref{gn}. And
our main tools will be the use of operator ideals (see \cite{Pie78,
Pie87}), the basic estimates of related widths (see Gluskin
\cite{Gl83}, Skrzypczak and Vyb\'iral \cite{Sk05,SV09,Vy08}) and
their relations to entropy numbers due to Carl \cite{Car81} and
Cobos and K$\ddot{\rm u}$hn \cite{CK09,Ku08}.

\begin{no}
By the symbol ` $\hookrightarrow$'  we denote continuous embeddings.

By $\mathbb{N}$ we denote the set of natural numbers, by\
$\mathbb{N}_0$\ the set\, $\mathbb{N}\cup\{0\}$.

Identity operators will always be denoted by {\rm id}. Sometimes we
do not indicate the spaces where {\rm id} is considered, and
likewise for other operators.

Let $X$ and $Y$ be complex quasi-Banach spaces and denote by
$\mathcal {L}(X, Y)$ the class of all linear continuous operators
$T:\,X \rightarrow\, Y.$ If no ambiguity arises, we write $\|T\|$
instead of the more exact versions $\|T ~|~ \mathcal {L}(X, Y)\|$ or
$\|T:X\rightarrow Y\|$.

The symbol $x_k \preceq y_k$ means that there exists a constant $c
> 0$\ such
that\ $x_k\le c\,y_k$\ for all\ $k\in\mathbb{N}.$\ And $x_k \succeq
y_k$ stands for $y_k \preceq x_k,$\ while $x_k\sim y_k$ denotes\
$x_k\preceq y_k \preceq x_k.$

All unimportant constants will be denoted by $c$ or $C$, sometimes
with additional indices.
\end{no}

Now we recall the definitions of the approximation, Gelfand and
Kolmogorov numbers (see \cite{Pie78, Pin85}).  We use the symbol
$A\subset\subset B$ if $A$ is a closed subspace of a topological
vector space $B$.

\begin{definition}
Let $T \in\mathcal {L}(X, Y)$\ and\ $k\in \mathbb{N}$.
\begin{enumerate}
\item[{\rm (i)}]\
The {\rm $k$th approximation number}\, of~ $T$ is defined by
\begin{equation*}
a_k(T, \,X, \,Y)=\inf\{\|T - A\|:~ A\in \mathcal{L}(X,Y) ~\,{\rm
with~ rank} (A) < k\},
\end{equation*}
also written by $a_k(T)$ if no confusion is possible. Here ${\rm
rank} (A)$ is the dimension of the range of the operator $A$.
\item[{\rm (ii)}]\
The {\rm $k$th Kolmogorov number}\, of~ $T$ is defined by
\begin{equation*}
d_k(T, X, Y)=\inf\{\|Q_N^YT\|:\,N\subset\subset Y,\,{\rm dim}
(N)<k\},
\end{equation*}
also written by $d_k(T)$ if no confusion is possible. Here, $Q_N^Y$
stands for the natural surjection of\,\,\,$Y$ onto the quotient
space $Y/N$.
\item[{\rm (iii)}]\
The {\rm $k$th Gelfand number}\, of~ $T$ is  defined by
\begin{equation*}
c_k(T, X, Y)=\inf\{\|TJ_M^X\|:\,M\subset\subset X,\,{\rm codim}
(N)<k\},
\end{equation*}
also written by $c_k(T)$ if no confusion is possible. Here, $J_M^X$
stands for the natural injection of\,\,\,$M$ into $X$.
\end{enumerate}
\end{definition}

Note that the $k$-th approximation, Kolmogorov and Gelfand number
are identical to the $(k-1)$-th linear, Kolmogorov and Gelfand width
of $T$, respectively, see Pinkus \cite{Pin85}.

It is well-known that the operator $T$ is compact if and only if
$\lim_k d_k(T)=0$ or equivalently $\lim_k c_k(T)=0$, but if $\lim_k
a_k(T)=0$, see \cite{Pin85}. The opposite implication for $a_k(T)$
is not true in general.

Both concepts, Kolmogorov and Gelfand numbers, are related to each
other. Namely they are dual to each other in the following sense,
cf. \cite{Pie78, Pin85}: If $X$ and $Y$ are Banach spaces, then
\begin{equation}\label{dualc*d}
c_k(T^\ast)=d_k(T)
\end{equation}
for all compact operators $T\in\mathcal{L}(X, Y)$ and
\begin{equation}\label{duald*c}
d_k(T^\ast)=c_k(T)
\end{equation}
for all $T\in\mathcal{L}(X, Y).$

Approximation, Gelfand and Kolmogorov numbers are subadditive and
multiplicative s-numbers. One may consult Pietsch
\cite{Pie87}(Sections 2.4, 2.5), for the proof in the Banach space
case. Further, the generalization to $p$-Banach spaces follows
obviously. Let $Y$ be a $p$-Banach space,\ $0<p\le 1$. And let $s_k$
denote any of the three quantities\ $a_k,\, d_k$ or $c_k$. More
precisely, we collect several common properties of them as follows,
\vspace{-0.2cm}
\begin{enumerate}
\item[]{\rm\bf(PS1)}\ (monotonicity)\,
$\|T\|=s_1(T)\ge s_2(T)\ge\cdots\ge 0$ for all $T\in\mathcal{L}(X,
Y)$,\vspace{-0.2cm}

\item[]{\rm\bf(PS2)}\ (subadditivity)\, $s_{m+k-1}^p(S+T)\leq
s_m^p(S)+s_k^p(T)$\, for all $m, k\in\mathbb{N},\,\,S,
T\in\mathcal{L}(X, Y)$,\vspace{-0.2cm}

\item[]{\rm\bf(PS3)}\ (multiplicativity)\, $s_{m+k-1}(ST)\leq
s_m(S)s_k(T)$\, for all $T\in\mathcal{L}(X, Y)$, $S\in\mathcal{L}(Y,
Z)$

\quad\quad and $m, k\in\mathbb{N},$\, cf. \cite{Pie78}(p. 155),
where $Z$ denotes a quasi-Banach space,\vspace{-0.2cm}

\item[]{\rm\bf(PS4)}\ (rank property)\, ${\rm rank}(T)<k$ if and only if
$s_k(T)=0$, where $T\in\mathcal{L}(X, Y)$.
\end{enumerate}

Moreover, there exist the following relationships:
\begin{equation}\label{acd}
 a_k(T)\ge \max( c_k(T), d_k(T)),\ \ \ k \in \mathbb{N}.
\end{equation}

We organize this paper as follows. In Section 2, we introduce
weighted function spaces of B-type and F-type, and present our main
results. In Section 3, we study the approximation numbers of
embeddings of related sequence spaces and use these results to
derive the desired estimates for the function space embeddings under
consideration. Similar results on the Gelfand and Kolmogorov numbers
of such embeddings are also established. Finally, Section 4 is
devoted to the investigation of widths for two typical classes of
weights and complementing a related result on a special case,
$\alpha = d / {p^*}>0$, which appeared in \cite{Ku08, KLSS06}.

\section{Widths of weighted function spaces}
\subsection{Weighted function spaces}
Throughout this paper we are interested in the function spaces with
a general class of weights, which satisfy two different types of
conditions. In order to investigate the behavior of the weight near
infinity, we may ignore local singularities of the weight and
concentrate on smooth weights. And it will be convenient for us to
define the following class $\mathcal{W}_1$ of weights, see
\cite{ET96}.
\begin{definition}
We say that a function $w ~ : \mathbb{R}^d\rightarrow (0,\infty)$
belongs to $\mathcal{W}_1$ if it satisfies the following conditions.
\\{\rm (i)}~ The function $w$ is infinitely differentiable. \\
{\rm (ii)}~ There exist a constant $c > 0$ and a number $\alpha\ge
0$ such that
\begin{equation}\label{w1}
0<w(x)\leq cw(y)(1+|x-y|)^\alpha
\end{equation}
holds for all $x, y\in\mathbb{R}^d$.
\\{\rm (iii)}~ For all multi-indices $\alpha \in \mathbb{N}_0^d$ the quantities
 $$c_{w,\alpha} := \sup\limits_{x\in\mathbb{R}^d}
 \frac{ |D^\alpha w(x)|}{w(x)}$$
are finite.
\end{definition}

We suppose that the reader is familiar with the (unweighted) Besov
and Triebel-Lizorkin spaces on $\mathbb{R}^d$. One can consult
\cite{ET96,Tr83,Tr06} and many other literatures for the definitions
and basic properties. As usual, $\mathcal{S}^\prime(\mathbb{R}^d)$
denotes the set of all tempered distributions on the Euclidean
d-space $\mathbb{R}^d$. For us it will be convenient to introduce
weighted function spaces to be studied here.

\begin{definition}\label{BF}
Let $0< p, q\leq \infty,\,$ and $s\in\mathbb{R}$, and let $w\in
\mathcal{W}_1.$ Then we put
\begin{equation*}
B_{p,q}^{s}(\mathbb{R}^d, w)=\left\{f\in \mathcal{S}^\prime
(\mathbb{R}^d)\,:\, \|f ~|~ B_{p,q}^{s}(\mathbb{R}^d, w)\|=\|fw ~|~
B_{p,q}^{s}(\mathbb{R}^d)\|<\infty\right\},
\end{equation*}
\begin{equation*}
F_{p,q}^{s}(\mathbb{R}^d, w)=\left\{f\in \mathcal{S}^\prime
(\mathbb{R}^d)\,:\, \|f ~|~ F_{p,q}^{s}(\mathbb{R}^d, w)\|=\|fw ~|~
F_{p,q}^{s}(\mathbb{R}^d)\|<\infty\right\},
\end{equation*}
with $p<\infty$ for the $F$-spaces.
\end{definition}
\begin{re}
If no ambiguity arises, then we can write $B_{p,q}^{s}(w)$ and
$F_{p,q}^{s}(w)$ for brevity.
\end{re}
\begin{re}
There are different ways to introduce weighted function spaces; cf.,
eg., Edmunds and Triebel \cite{ET96}, or Schmeisser and Triebel
\cite{ST87}. One can also consult \cite{ KLSS05, Tr06} for related
remarks.
\end{re}

\subsection{Besov spaces and sequence spaces}
~~~~There are various ways to transform our problem in function
spaces to the simpler context of sequence spaces. Here we are going
to use the wavelet representations of Besov spaces as an essential
tool. We quote the wavelet characterization of weighted Besov spaces
proved in \cite{HT05} where quasi-Banach case is included, cf. also
\cite{KLSS05,KLSS06}.

\begin{prop}\label{Besov_des}
Let $s\in \mathbb{R}$ and $0< p,q\leq\infty.$ Assume $$r>\max(s,
\frac{2d}p+\frac d2-s).$$ Then for every weight $w_\alpha$ there
exists an orthonormal basis of compactly supported wavelet functions
$\{\varphi_{j,k}\}_{j,k}\cup\{\psi_{i,j,k}\}_{i,j,k},\
j\in\mathbb{N}_0,\ k\in \mathbb{Z}^d$ and $i=1, \ldots, 2^d-1$, such
that a distribution $f\in \mathcal {S}^\prime(\mathbb{R}^d)$ belongs
to $B_{p,q}^{s}(w_\alpha)$ if and only if
\begin{equation}\label{Besov_ell}
\begin{split}
\|f|B_{p,q}^{s}(w_\alpha)\|^\clubsuit =&\Big(
\sum\limits_{k\in\mathbb{Z}^d}|\langle f,\varphi_{0,k}\rangle
w_\alpha(k)|^p \Big)^{1/p}
\\
&+\sum\limits_{i=1}^{2^d-1}\Big\{ \sum\limits_{j=0}^\infty
2^{j\big(s+d(\frac 12-\frac 1p)\big)q}
\Big(\sum\limits_{k\in\mathbb{Z}^d}|\langle f,\psi_{i,j,k}\rangle
 w_\alpha(2^{-j}k)|^p\Big)^{q/p}\Big\}^{1/q}<\infty.
\end{split}
\end{equation}
Furthermore, $\|f|B_{p,q}^{s}(w_\alpha)\|^\clubsuit$ may be used as
an equivalent quasi-norm in $B_{p,q}^{s}(w_\alpha)$.
\end{prop}

\begin{re}
The proof of this proposition may be found in Haroske and Triebel
\cite{HT05}. One can also consult \cite{KLSS05} for historical
remarks.
\end{re}
Let $0< p,q\leq\infty.$ Motivated by Proposition \ref{Besov_des} we
will work with the following weighted sequence spaces
\begin{equation}\label{ellal}
\begin{split}
\ell_q(2^{js}\ell_p(w)):=\Bigg\{&
\lambda=(\lambda_{j,k})_{j,k}:~~\lambda_{j,k}\in\mathbb{C},\\
&\|\lambda|\ell_q(2^{js}\ell_p(w))\|= \Big( \sum\limits_{j=0}^\infty
2^{jsq}
\Big(\sum\limits_{k\in\mathbb{Z}^d}|\lambda_{j,k}\,w_{j,k}|^p
\Big)^{q/p}\Big)^{1/q}<\infty \Bigg\},
\end{split}
\end{equation}
(usual modification if $p=\infty$ and/or $q=\infty$), where
$w_{j,k}=w(2^{-j}k).$ If $s=0$ we will write $\ell_q(\ell_p(w))$. In
contrast to the norm defined in (\ref{Besov_ell}), the finite
summation on $i=1, 2, \ldots, 2^d-1$ is irrelevant and can be
dropped out.
\subsection{Main assertions}
For later estimates of widths of embeddings between related sequence
spaces, we introduce a second class of weights as follows.
\begin{definition}
We shall say, a measurable function
$\varphi\,:\,[1,\infty)\,\rightarrow\,(0,\infty)$ belongs to
$\mathcal{V}$ if the inequalities
\begin{equation}\label{def v}
0<\underline{\varphi}(t):=\inf_{s\in[1,\infty)}\frac{\varphi(ts)}{\varphi(s)},\quad\quad
\bar{\varphi}(t):=\sup_{s\in[1,\infty)}\frac{\varphi(ts)}{\varphi(s)}<\infty
\quad {\rm for\, all}~t\in[1,\infty),
\end{equation}
are satisfied and if $\underline{\varphi}$ and $\bar{\varphi}$ are
measurable.
\end{definition}
Now we recall two indices necessary for subsequent discussion as
follows,

$$\alpha_\varphi:=\inf\limits_{t>1}\frac{\log \bar{\varphi}(t)}{\log
t},\quad \beta_\varphi:=\sup\limits_{t>1}\frac{\log
\underline{\varphi}(t)}{\log t}.$$

\begin{definition}\label{w2}
A function $w :\,\mathbb{R}^d\,\rightarrow\,(0,\infty)$ belongs to
the class $\mathcal {W}_2$ if there exist positive constants $c_1$
and $c_2$ and a function $\varphi \in \mathcal{V}$ such that
\begin{equation*}
\begin{array}{cl}
c_1\,\leq\,w(x)\,\leq\,c_2, &{\rm for~ all}\,\, x,|x|\leq 1,\\
c_1\varphi(|x|)\,\leq\,w(x)\,\leq\,c_2\varphi(|x|),\quad &{\rm for~
all}\,\, x,|x|> 1.
\end{array}
\end{equation*}
We say that a function $\varphi \in \mathcal{V}$ is `associated
with' $w$ if these inequalities are satisfied.
\end{definition}
We define
$$\frac 1{p^*}:=\Big(\frac 1{p_2}-\frac 1{p_1}\Big)_+\ \ \ {\rm and}
\ \ \ \frac 1p:=\frac 1{p_2}-\frac 1{p_1}.$$ Then $\delta>d/p$ holds
obviously. The following characterization of the compactness of
embeddings in the general situation was described in \cite{KLSS06}.
\begin{prop}
Let $w\in \mathcal{W}_1 \cap\mathcal{W}_2$, and let $\varphi \in
\mathcal{V}$ be an associated function in the sense of Definition
\ref{w2}.\vspace{-0.2cm}
 \begin{enumerate}
\item[$(i)$]\ Let $p^*=\infty$. Then the embedding
\begin{equation}\label{bb}
B_{p_1,q_1}^{s_1}(\mathbb{R}^d, w)\hookrightarrow
B_{p_2,q_2}^{s_2}(\mathbb{R}^d)
\end{equation}
is compact if and only if $\delta> 0$ and $\lim_{t\rightarrow\infty}
\varphi(t)=\infty$.\vspace{-0.2cm}
\item[$(ii)$]\ Let $0<p^*<\infty$. Then the embedding (\ref{bb}) is compact if
and only if
$$\delta > d / {p^*} ~ {\rm and} ~ \int_1^\infty \varphi(t)^{-p^*}t^d\frac{dt}t
<\infty.$$
 \end{enumerate}
\end{prop}

For $0<p\le\infty,$ we set\vspace{0.2cm}

$p^\prime=
\begin{cases}
\frac p{p-1}\quad &{\rm if}\ 1<p<\infty,\\
1 &{\rm if}\ p=\infty,\\
\infty &{\rm if}\ 0<p\le 1.
\end{cases}
$\vspace{0.2cm}

We are now ready to formulate our main assertions.
\begin{theorem}\label{an}
 Let $w\in \mathcal{W}_1
\cap\mathcal{W}_2$, and let $\varphi \in \mathcal{V}$ be an
associated function in the sense of Definition \ref{w2}. Suppose
$\frac 1{\tilde{p}}=\frac\mu d+\frac 1{p_1},$ where
$\mu=\min(\beta_\varphi,\delta)$. Denote by $a_k$ the $k$th
approximation number of the Sobolev embedding (\ref{bb}).
\vspace{-0.2cm}
 \begin{enumerate}
\item[$(i)$]\
Suppose in addition that $0< p_1 \le p_2 \le 2$\,\,or\,\,$2\le
p_1\le p_2\le \infty.$ Then
$$
 a_k \sim
\begin{cases}
 \varphi\big(k^{\frac 1d}\big)^{-1} ~ & {\rm if} ~
0<\beta_\varphi\le\alpha_\varphi <\delta,
\\ k^{-\frac\delta d} ~ & {\rm if} ~ 0 <\delta<\beta_\varphi.
\end{cases}
$$
\vspace{-0.6cm}
\item[$(ii)$]\ Suppose that in addition to the general hypotheses,
$\tilde{p} < p_2 < p_1 \le \infty.$ Then
$$
a_k \sim
\begin{cases}
k^{\frac 1{p_2} - \frac 1{p_1}}\varphi\big(k^{\frac 1d}\big)^{-1} &
{\rm if} ~ \frac dp<\beta_\varphi\le\alpha_\varphi
<\delta,\\
 k^{-\frac\delta d+\frac 1{p_2} - \frac 1{p_1}} ~ & {\rm if} ~
\frac dp < \delta < \beta_\varphi.\end{cases}$$ \vspace{-0.6cm}
\item[$(iii)$]\ Suppose that
in addition to the general hypotheses, $0< p_1 < 2 < p_2\le \infty$.
Then
$$a_k \sim
\begin{cases}
 k^{\frac 1t - \frac 12}\varphi\big(k^{\frac 1d}\big)^{-1}
 & {\rm if} ~ \frac dt<\beta_\varphi\le\alpha_\varphi <\delta,
\\
 k^{-\frac\delta d+\frac 1t - \frac 12} ~& {\rm if} ~ \frac
dt < \delta < \beta_\varphi,
\\
 \varphi\big(k^{\frac t{2d}}\big)^{-1} ~& {\rm if} ~
0<\beta_\varphi\le\alpha_\varphi < \frac dt~ {\rm and}
~\alpha_\varphi <  \delta,
\\
 k^{-\frac t{2d}\delta} ~& {\rm if} ~  \delta
<\beta_\varphi~ {\rm and} ~\delta <\frac dt,
\end{cases}
$$ where
$t=\min(p_1^\prime,p_2)$.
\end{enumerate}
\end{theorem}

\begin{theorem}\label{kn}
Let $w\in \mathcal{W}_1 \cap\mathcal{W}_2$, and let $\varphi \in
\mathcal{V}$ be an associated function in the sense of Definition
\ref{w2}. Assume $\theta = \frac{1/{p_1}-1/{p_2}}{1/2-1/{p_2}}$ and
$\frac 1{\tilde{p}}=\frac\mu d+\frac 1{p_1},$ where
$\mu=\min(\beta_\varphi,\delta)$. Denote by $d_k$ the $k$th
approximation number of the Sobolev embedding (\ref{bb}).
\vspace{-0.2cm}
 \begin{enumerate}
\item[$(i)$]\ Suppose in addition that $0< p_1 \le p_2 \le
2$\,\,or\,\,$2< p_1= p_2\le \infty.$ Then
$$
d_k \sim
\begin{cases}
\varphi\big(k^{\frac 1d}\big)^{-1} ~ &{\rm if} ~
0<\beta_\varphi\le\alpha_\varphi <\delta,\\
 k^{-\frac\delta d} ~ &{\rm if} ~ 0 <\delta<\beta_\varphi.
\end{cases}
$$
\vspace{-0.6cm}

\item[$(ii)$]\ Suppose that in addition to the general hypotheses,
$\tilde{p} < p_2 < p_1 \le \infty.$ Then
$$
d_k \sim
\begin{cases}
 k^{\frac 1{p_2} - \frac 1{p_1}}\varphi\big(k^{\frac
1d}\big)^{-1} ~ &{\rm if} ~ \frac dp<\beta_\varphi\le\alpha_\varphi
<\delta,\\
 k^{-\frac\delta d+\frac 1{p_2} - \frac 1{p_1}} ~ &{\rm if}
~ \frac dp < \delta < \beta_\varphi.
\end{cases}
$$
\vspace{-0.6cm}

\item[$(iii)$]\ Suppose that in addition to the general hypotheses, $0< p_1
< 2 < p_2\le \infty.$ Then
$$
d_k \sim
\begin{cases}
 k^{\frac 1{p_2} - \frac 12}\varphi\big(k^{\frac
1d}\big)^{-1} ~ &{\rm if} ~ \frac
d{p_2}<\beta_\varphi\le\alpha_\varphi <\delta,
\\
 k^{-\frac\delta d+\frac 1{p_2} - \frac 12} ~ &{\rm if} ~ \frac
d{p_2} < \delta < \beta_\varphi,
\\
 \varphi\big(k^{\frac {p_2}{2d}}\big)^{-1} ~ &{\rm if} ~
0<\beta_\varphi\le\alpha_\varphi < \frac d{p_2}$~ {\rm and}
~$\alpha_\varphi < \delta,
\\
 k^{-\frac {p_2}{2d}\delta} ~ &{\rm if} ~  \delta
<\beta_\varphi~ {\rm and} ~\delta <\frac d{p_2}.
\end{cases}
$$
\vspace{-0.6cm}

\item[$(iv)$]\ Suppose that in addition to the general hypotheses, $2\le p_1
< p_2\le \infty.$ Then
$$
d_k \sim
\begin{cases}
 k^{\frac 1{p_2} - \frac 1{p_1}}\varphi\big(k^{\frac
1d}\big)^{-1} ~ &{\rm if} ~ \frac
d{p_2}\theta<\beta_\varphi\le\alpha_\varphi <\delta,
\\
k^{-\frac\delta d+\frac 1{p_2} - \frac 1{p_1}} ~ &{\rm if} ~ \frac
d{p_2}\theta < \delta < \beta_\varphi,
\\
 \varphi\big(k^{\frac {p_2}{2d}}\big)^{-1} ~ &{\rm if} ~
0<\beta_\varphi\le\alpha_\varphi < \frac d{p_2}\theta~ {\rm and}
~\alpha_\varphi < \delta,
\\
 k^{-\frac {p_2}{2d}\delta}~ &{\rm if} ~  \delta
<\beta_\varphi~ {\rm and}~\delta <\frac d{p_2}\theta.
\end{cases}
$$\vspace{-0.2cm}
\end{enumerate}
\end{theorem}

\begin{theorem}\label{gn}
Let $w\in \mathcal{W}_1 \cap\mathcal{W}_2$, and let $\varphi \in
\mathcal{V}$ be an associated function in the sense of Definition
\ref{w2}. Assume $\theta_1 = \frac{1/{p_1}-1/{p_2}}{1/{p_1}-1/2},\,$
and $\frac 1{\tilde{p}}=\frac\mu d+\frac 1{p_1}$ where
$\mu=\min(\beta_\varphi,\delta)$. Denote by $c_k$ the $k$th
approximation number of the Sobolev embedding (\ref{bb}).
\vspace{-0.2cm}
 \begin{enumerate}
\item[$(i)$]\ Suppose in addition that $2\le p_1\le p_2\le
\infty$\,\,or\,\,$0< p_1 = p_2< 2.$ Then
$$
c_k \sim
\begin{cases}
 \varphi\big(k^{\frac 1d}\big)^{-1} ~ &{\rm if} ~
0<\beta_\varphi\le\alpha_\varphi <\delta,
\\
 k^{-\frac\delta d} ~ &{\rm if} ~ 0 <\delta<\beta_\varphi.
\end{cases}$$
\vspace{-0.6cm}

\item[$(ii)$]\ Suppose that in addition to the general hypotheses,
$\tilde{p} < p_2 < p_1 \le \infty.$ Then
$$
c_k \sim
\begin{cases}
 k^{\frac 1{p_2} - \frac 1{p_1}}\varphi\big(k^{\frac
1d}\big)^{-1} ~ &{\rm if} ~ \frac dp<\beta_\varphi\le\alpha_\varphi
<\delta,
\\
 k^{-\frac\delta d+\frac 1{p_2} - \frac 1{p_1}} ~ &{\rm if}
~ \frac dp < \delta < \beta_\varphi.
\end{cases}
$$
\vspace{-0.6cm}

\item[$(iii)$]\ Suppose that in addition to the general hypotheses, $0< p_1
< 2 < p_2\le \infty.$ Then
$$
c_k \sim
\begin{cases}
 k^{\frac 12 - \frac 1{p_1}}\varphi\big(k^{\frac
1d}\big)^{-1} ~ &{\rm if} ~ \frac
d{p_1^\prime}<\beta_\varphi\le\alpha_\varphi <\delta,
\\
 k^{-\frac\delta d+\frac 12 - \frac 1{p_1}} ~ &{\rm if} ~
\frac d{p_1^\prime} < \delta < \beta_\varphi,
\\
 \varphi\big(k^{\frac {p_1^\prime}{2d}}\big)^{-1} ~ &{\rm
if} ~ 0<\beta_\varphi\le\alpha_\varphi < \frac d{p_1^\prime}~ {\rm
and} ~\alpha_\varphi < \delta,
\\
 k^{-\frac {p_1^\prime}{2d}\delta} ~ &{\rm if} ~  \delta
<\beta_\varphi~ {\rm and} ~\delta <\frac d{p_1^\prime}.
\end{cases}
$$
\vspace{-0.6cm}

\item[$(iv)$]\ Suppose that in addition to the general hypotheses, $0< p_1 <
p_2 \le 2.$ Then
$$
c_k \sim
\begin{cases}
 k^{\frac 1{p_2} - \frac 1{p_1}}\varphi\big(k^{\frac
1d}\big)^{-1} ~ &{\rm if} ~ \frac
d{p_1^\prime}\theta_1<\beta_\varphi\le\alpha_\varphi <\delta,
\\
k^{-\frac\delta d+\frac 1{p_2} - \frac 1{p_1}} ~ &{\rm if} ~ \frac
d{p_1^\prime}\theta_1 < \delta < \beta_\varphi,
\\
 \varphi\big(k^{\frac {p_1^\prime}{2d}}\big)^{-1} ~ &{\rm
if} ~ 0<\beta_\varphi\le\alpha_\varphi < \frac
d{p_1^\prime}\theta_1~ {\rm and} ~\alpha_\varphi < \delta,
\\
k^{-\frac {p_1^\prime}{2d}\delta} ~ &{\rm if} ~ \delta
<\beta_\varphi~ {\rm and} ~\delta <\frac d{p_1^\prime}\theta_1.
\end{cases}
$$\vspace{-0.2cm}
 \end{enumerate}
\end{theorem}

\begin{re}
We shift the proofs of the above three theorems to Subsection
\ref{akg}. We wish to mention that all of the results herein do not
depend on the microscopic parameters\ $q_1$ and\ $q_2$.
\end{re}

\begin{re}
From Definition \ref{BF}, we know that an operator $f\mapsto w f$ is
an isomorphic mapping from $B_{p,q}^{s}(\mathbb{R}^d, w)$ onto
$B_{p,q}^{s}(\mathbb{R}^d)$. And it holds that
\begin{equation*}
s_k({\rm id}, ~B_{p_1,q_1}^{s_1}(\mathbb{R}^d, w_1),~
B_{p_2,q_2}^{s_2}(\mathbb{R}^d, w_2)) ~ \sim ~ s_k({\rm id},
~B_{p_1,q_1}^{s_1}(\mathbb{R}^d, w_1/w_2),~
B_{p_2,q_2}^{s_2}(\mathbb{R}^d)),
\end{equation*}
where\ $w_1$\ and $w_2$ are admissible weight functions, and $s_k$
denotes any of the three quantities\ $a_k,\, d_k$ or $c_k$.
Therefore, without loss of generality we can assume that the target
space is an unweighted space.
\end{re}

\section{Widths in sequence spaces}
In this section, we give the proofs of our main assertions stated
above in terms of embeddings of related sequence spaces.
\subsection{Approximation numbers of sequence spaces}\label{ans}
To begin with, we shall recall some lemmata. Lemma \ref{an1} follows
trivially from results of Gluskin \cite{Gl83} and Edmunds and
Triebel \cite{ET96}.
\begin{lemma}\label{an1}
 Let $N\in\mathbb{N}$ and $k\le\frac{N}{4}$.
\\ $(i)$ If $0< p_1\le p_2\le 2$~or~$2\le p_1\le p_2\le \infty$~ then
$$a_{k}\big({\rm id}, \ell_{p_1}^N, \ell_{p_2}^N\big)\sim 1.$$
$(ii)$ If $1\le p_1 < 2 < p_2\le \infty$ and $(p_1, p_2) \neq (1,
\infty)$\, then
$$a_{k}\big({\rm id}, \ell_{p_1}^N, \ell_{p_2}^N\big)\sim \min\big(1, N^{1/t}k^{-1/2}\big).$$
where $\frac 1t= \frac 1{\min(p_1^\prime, p_2)}.$
\\ $(iii)$ If
$0<  p_1< 1$ and $ 2 < p_2< \infty$\, then
$$a_{k}\big({\rm id}, \ell_{p_1}^N, \ell_{p_2}^N\big)\sim
\min\big(1, N^{1/{p_2}}k^{-1/2}\big).$$
\end{lemma}
Lemma \ref{an2} in the case $1\le p_2 < p_1\le \infty$ may be found
in Pietsch \cite{Pie78}, Section 11.11.5, also in Pinkus
\cite{Pin85}(p.\,203). The proof may be directly generalized to the
quasi-Banach setting $0< p_2 < p_1\le \infty.$

\begin{lemma}\label{an2}
Let $0<  p_2 < p_1\leq \infty$ and $k\le N$. Then
$$a_{k}\left({\rm id}, \ell_{p_1}^N, \ell_{p_2}^N\right)
=(N - k + 1)^{1/{p_2}-1/{p_1}}.$$
\end{lemma}

The following lemma is taken from  \cite{Sk05}.
\begin{lemma}\label{an3}  Assume that $1\le p_1 < 2 < p_2< \infty$
and $(p_1, p_2) \neq (1, \infty)$. Then there is a positive constant
$C$ independent of $N$ and $k$ such that
\begin{equation}\label{alp2p}
a_{k}\Big({\rm id}, ~\ell_{p_1}^N, ~\ell_{p_2}^N\Big)\leq C\,
\left\{
\begin{array}{ll}
1 & {\rm if}\, k\leq N^{2/t},\\
N^{1/t}k^{-1/2}\quad & {\rm if}\, N^{2/t}< k \leq
N,\\
0 & {\rm if}\, n>N,
\end{array}\right.
\end{equation}
where $\frac 1t= \frac 1{\min(p_1^\prime, p_2)}.$
\end{lemma}

\begin{lemma}\label{an1inf}
Let $0 < p\le 1$ and $N\in\mathbb{N}$. \vspace{-0.2cm}
\begin{enumerate}
\item[$(i)$] Let $0<\lambda<1$. Then there exists a constant
$C_\lambda>0$ depending only on $\lambda$ such that
\begin{equation}\label{an1infupp}
a_k\big({\rm id}, \ell_{p}^N, \ell_\infty^N\big)\le
\begin{cases}
1 & {\rm if}\ \ k\leq N^\lambda,\\
C_\lambda n^{-1/2}\quad & {\rm if}\ \ N^\lambda<k\le N,\\
0 & {\rm if}\ \ k>N. \end{cases}
\end{equation}
\vspace{-0.6cm}
\item[$(ii)$] There exists a constant
$C>0$ independent of $k$ such that for any  $k\in \mathbb{N}$
\begin{equation}\label{an1inflow}
a_{k}\big({\rm id}, \ell_{p}^{2k}, \ell_\infty^{2k}\big)\ge C
k^{-1/2}.
\end{equation}
\end{enumerate}
\end{lemma}
We refer to \cite{Vy08} for the proof.

Following Pietsch \cite{Pie87}, we associate to the sequence of the
$s$-numbers the following operator ideals, and for $0<r<\infty$, we
put
\begin{equation}\mathscr{L}_{r,\infty}^{(s)}:=\left\{T\in\mathcal{L}(X,
Y):\quad \sup\limits_{k\in\mathbb{N}}k^{1/r}s_k(T)<\infty\right\}.
\end{equation}
 Equipped with the quasi-norm
\begin{equation}\label{idealsdef}
L_{r,\infty}^{(s)}(T):=\sup\limits_{k\in\mathbb{N}}k^{1/r}s_k(T),
\end{equation}
the set $\mathscr{L}_{r,\infty}^{(s)}$ becomes a quasi-Banach space.
For such quasi-Banach spaces there always exists a real number
$0<\rho\leq 1$ such that
\begin{equation}\label{idealsinq}
L_{r,\infty}^{(s)}\left(\sum\limits_jT_j\right)^\rho\leq
\sum\limits_jL_{r,\infty}^{(s)}(T_j)^\rho
\end{equation}
holds for any sequence of operators
$T_j\in\mathscr{L}_{r,\infty}^{(s)}.$ Then we shall use the
notations $L_{r,\infty}^{(a)},~L_{r,\infty}^{(c)}$ and
$L_{r,\infty}^{(d)}$ for the approximation, Gelfand and Kolmogorov
numbers, respectively.

Next, we collect two lemmata about some useful estimates of certain
sums, which play key roles on the sequence space level later on.
Lemma \ref{ab} may be found in \cite{KLSS06}. And the proof of Lemma
\ref{rab} follows literally as in the proof of Lemma 4.9 in
\cite{KLSS06} with a simple substitution.

\begin{lemma}\label{ab}
Assume $\varphi\in\mathcal {V}$.
\\
$(i)$ For all $s\ge 1$ and $t\ge 1$ we have
\begin{equation}\label{phits}
\underline{\varphi}(t)\varphi(s) < \varphi(ts) \le
\bar{\varphi}(t)\varphi(s).
\end{equation}
$(ii)$ The following inequality holds:
\begin{equation}
-\infty < \beta_\varphi \le \alpha_\varphi < \infty.
\end{equation}
$(iii)$ For any $\varepsilon>0$ there exists a constant
$c=c(\varepsilon)\ge 1$ such that
$$c^{-1}s^{\beta_\varphi-\varepsilon}\le \frac {\varphi(s)}{\varphi(1)}
\le c s^{\alpha_\varphi+\varepsilon}\quad\quad for ~ all ~ s\ge 1.$$
\end{lemma}

\begin{lemma}\label{rab}
Let $\rho>0$ and $\varphi\in\mathcal {V}$.
\\ $(i)$ If $\gamma>\alpha_\varphi+\eta,$\ then
\begin{equation}\label{ar}
\sup\limits_{M\ge1}2^{M\eta\rho}\varphi(2^M)^\rho\sum\limits_{m=0}^M
2^{-m\gamma\rho}2^{-(M-m)\eta\rho}\varphi(2^{M-m})^{-\rho}<\infty.
\end{equation}
\\ $(ii)$ If $\gamma<\beta_\varphi+\eta,$\ then
\begin{equation}\label{rb1}
\sup\limits_{M\ge1}2^{M\gamma\rho}\sum\limits_{m=0}^M
2^{-m\gamma\rho}2^{-(M-m)\eta\rho}\varphi(2^{M-m})^{-\rho}<\infty
\end{equation}
and
\begin{equation}\label{rbm}
\sup\limits_{M\ge1}\varphi(2^M)^\rho2^{-M\gamma\rho}\sum\limits_{m=M+1}^\infty
2^{m\gamma\rho}\varphi(2^m)^{-\rho}<\infty.
\end{equation}
\end{lemma}
\begin{re}
The relation given in (\ref{ar}) may also be represented simply as
\begin{equation*}
\sup\limits_{M\ge1}\varphi(2^M)^\rho\sum\limits_{m=0}^M
2^{m(\eta-\gamma)\rho}\varphi(2^{M-m})^{-\rho}<\infty,
\end{equation*}
\end{re}
which is essentially the same as (4.6) in Lemma 4.6 from
\cite{KLSS06}.

\begin{prop}\label{an4}
Suppose $1\le p_1 < 2 < p_2 \le \infty,$\,and $(p_1, p_2)\neq(1,
\infty)$. Set $t=\min(p_1^\prime,p_2)$. Let $w\in \mathcal{W}_1
\cap\mathcal{W}_2$, and let $\varphi \in \mathcal{V}$ be an
associated function in the sense of Definition \ref{w2}. Then
\begin{equation}
a_{k}\Big({\rm id}, \ell_{q_1}(2^{j\delta}\ell_{p_1}(w)),
\ell_{q_2}(\ell_{p_2})\Big) \sim \left\{
\begin{array}{ll}
 k^{\frac 1t - \frac
12}\varphi\big(k^{\frac 1d}\big)^{-1} & {\rm if}\,~ \frac
dt<\beta_\varphi\le\alpha_\varphi <\delta,
\\
k^{-\frac\delta d+\frac 1t - \frac 12} & {\rm if}\,~ \frac dt < \delta < \beta_\varphi,\\
\varphi\big(k^{\frac t{2d}}\big)^{-1} & {\rm if}\,~
0<\beta_\varphi\le\alpha_\varphi < \frac dt~ and ~\alpha_\varphi <
\delta,
\\
k^{-\frac t{2d}\delta} & {\rm if}\,~ \delta <\beta_\varphi$~ and
$~\delta <\frac dt.
\end{array}\right.
\end{equation}
\end{prop}
\begin{proof}
{\tt Step 1}.\quad Preparations. We denote
$$\Lambda:=\{\lambda=(\lambda_{j,\ell}):
\,\quad\lambda_{j,\ell}\in\mathbb{C},\,\quad j\in
\mathbb{N}_0,\,\ell\in\mathbb{Z}^d\},$$ and
$$ B_1=\ell_{q_1}(2^{j\delta}\ell_{p_1}(w)),\,\quad {\rm
and}\,\quad  B_2= \ell_{q_2}(\ell_{p_2}).$$
Let\,$I_{j,i}\subset\mathbb{N}_0\times\mathbb{Z}^d$ be such that
\begin{equation}\label{Ij0}
I_{j,0}:=\{(j,\ell):\, |\ell|\leq 2^j\},\quad j\in\mathbb{N}_0,
\end{equation}
\begin{equation}\label{Iji}
I_{j,i}:=\{(j,\ell):\, 2^{j+i-1}<|\ell|\leq 2^{j+i}\}, \quad i\in
\mathbb{N},\quad j\in\mathbb{N}_0.
\end{equation} Besides, let
$P_{j,i}:\Lambda\mapsto\Lambda$ be the canonical projection with
respect to $I_{j,i}$, i.e., for $\lambda\in\Lambda$, we put
\begin{equation*}
(P_{j,i}\lambda)_{u,v}:= \left\{
\begin{array}{ll}
\lambda_{u,v}\quad & (u,v)\in I_{j,i},\\
0\quad & {\rm otherwise},
\end{array}\right. \quad u\in
\mathbb{N}_0,\quad v\in\mathbb{Z}^d,\quad i\ge 0.
\end{equation*}
Then
\begin{equation}\label{rankpiece}M_{j,i}:=|I_{j,i}|\sim2^{(j+i)d},\end{equation}
\begin{equation}{\rm id}_\Lambda=
\sum\limits_{j=0}^\infty\sum\limits_{i=0}^\infty P_{j,i}.
\end{equation}
Due to Lemma \ref{ab} we have
\begin{equation}
w(2^{-j}\ell) \sim \varphi(2^{-j}|\ell|) \sim \varphi(2^i)\quad {\rm
if}\quad (j,\ell)\in I_{j,i}, \quad i \ge 0.
\end{equation}
Thanks to simple monotonicity arguments and explicit properties of
the approximation numbers we obtain
\begin{equation}
\begin{array}{ll}\label{andiscr}
a_k(P_{j,i},B_1,B_2)&\leq  \frac 1{\inf_{\ell\in
I_{j,i}}w(2^{-j}\ell)}2^{-j\delta}a_k({\rm id}, \ell_{p_1}^{M_{j,i}}, \ell_{p_2}^{M_{j,i}})\\
&\leq c\frac 1{2^{j\delta}\varphi(2^i)}a_k({\rm id},
\ell_{p_1}^{M_{j,i}}, \ell_{p_2}^{M_{j,i}}).
\end{array}
\end{equation}
{\tt Step 2}.\quad The operator ideal comes into play. To shorten
notations we shall put $\frac 1s=\frac 1r+\frac 12$\, for any
$s>0.$\, By (\ref{idealsdef}) and (\ref{andiscr}), we have
\begin{equation}\label{idealdiscr}
L_{s,\infty}^{(a)}(P_{j,i})\leq c\frac
1{2^{j\delta}\varphi(2^i)}L_{s,\infty}^{(a)}({\rm id},
\ell_{p_1}^{M_{j,i}}, \ell_{p_2}^{M_{j,i}})
\end{equation} The known asymptotic
behavior of the approximation numbers $a_k({\rm
id},\ell_{p_1}^{N},\ell_{p_2}^{N})$, cf. (\ref{alp2p}), and
(\ref{rankpiece}) yield that
\begin{equation}\label{idealse2}
L_{2,\infty}^{(a)}({\rm id}, \ell_{p_1}^{M_{j,i}},
\ell_{p_2}^{M_{j,i}})\leq
C2^{d(j+i)/t},\indent\indent\indent\indent\indent\indent
\end{equation}
\begin{equation}\label{ideals2}
L_{s,\infty}^{(a)}({\rm id}, \ell_{p_1}^{M_{j,i}},
\ell_{p_2}^{M_{j,i}})\leq C2^{d(j+i)(\frac 1t+\frac
1r)}\quad\quad\,\, {\rm if} \quad \frac 1s>\frac 12.
\end{equation}
\\{\tt Step 3}.\quad The estimate of $a_k({\rm id}, B_1, B_2)$ from
above in the case $\mu=\min(\beta_\varphi,\delta)>\frac dt.$\, For
any given $M\in\mathbb{N}_0$,\, we put
\begin{equation}\label{PQ}
P:=\sum\limits_{m=0}^M\sum\limits_{j+i=m}P_{j,i}\quad\quad{\rm
and}\quad\quad
Q=\sum\limits_{m=M+1}^\infty\sum\limits_{j+i=m}P_{j,i}.
\end{equation}
{\tt Substep 3.1.}\quad Estimate of $a_k(P, B_1, B_2)$. Let $\frac
1s>\frac 12.$\,Then in view of
(\ref{idealsinq}),\,(\ref{idealdiscr}) and (\ref{ideals2}), we have
\begin{equation}
\begin{array}{ll}\label{idealPrho}
L_{s,\infty}^{(a)}(P)^\rho&\leq
\sum\limits_{m=0}^M\sum\limits_{j+i=m}L_{s,\infty}^{(a)}(P_{j,i})^\rho\\
&\leq c_1 \sum\limits_{m=0}^M\sum\limits_{j+i=m} \frac
1{2^{j\rho\delta}\varphi^\rho(2^i)}2^{\rho
md(\frac 1t+\frac 1r)}\\
&\leq c_1 \sum\limits_{m=0}^M 2^{\rho md(\frac 1t+\frac
1r)}\sum\limits_{j=0}^m \frac
1{2^{j\rho\delta}\varphi^\rho(2^{m-j})}
\\
&\leq c_2 \sum\limits_{m=0}^M 2^{\rho md(\frac 1t+\frac 1r)}\times
\left\{ \begin{array}{ll}
1/{\varphi^\rho(2^{m})}\quad & {\rm if} ~\alpha_\varphi< \delta,\\
2^{-m\rho\delta}\quad & {\rm if} ~\beta_\varphi> \delta.
\end{array}\right.
\end{array}
\end{equation}
First we consider the case $\alpha_\varphi< \delta$. We choose $r>0$
such that $d(\frac 1t+\frac 1r)> \alpha_\varphi.$ Then
\begin{equation}
L_{s,\infty}^{(a)}(P)^\rho\leq c_3 2^{\rho dM(\frac 1t+\frac
1r)}\sum\limits_{m=0}^M \frac{2^{-\rho d(M-m)(\frac 1t+\frac
1r)}}{\varphi^\rho(2^m)}
\end{equation}
and (\ref{ar}) yield
\begin{equation}\label{idealP}
L_{s,\infty}^{(a)}(P)\leq c\ \frac{2^{dM(\frac 1t+\frac
1r)}}{\varphi(2^M)}.
\end{equation}
Using (\ref{idealsdef}) and (\ref{idealP}), we get
\begin{equation}\label{a_2dMPvar}
a_{2^{dM}}(P, B_1, B_2)\leq c\  2^{dM(\frac 1t-\frac
12)}{\varphi(2^M)}^{-1}.
\end{equation}
Next we consider the case $\beta_\varphi> \delta$. We choose $r>0$
such that $d(\frac 1t+\frac 1r)> \delta.$ Then, we proceed in a
similar way as above and obtain by (\ref{rb1}) that
\begin{equation}\label{a_2dMPdelta}
a_{2^{dM}}(P, B_1, B_2)\leq c\  2^{dM(\frac 1t-\frac 12-\frac\delta
d)}.
\end{equation}
{\tt Substep 3.2.}\quad Estimate of $d_n(Q, B_1, B_2)$. In a similar
way to (\ref{idealPrho}), we obtain by (\ref{idealse2}) that
\begin{equation}\label{idealQrho}
L_{2,\infty}^{(a)}(Q)^\rho\leq c_1 \sum\limits_{m=M+1}^\infty
2^{\rho md/t}\times \left\{ \begin{array}{ll}
1/{\varphi^\rho(2^{m})}\quad & {\rm if} ~\alpha_\varphi< \delta,\\
2^{-m\rho\delta}\quad & {\rm if} ~\beta_\varphi> \delta.
\end{array}\right.
\end{equation}
For the case $\frac dt<\beta_\varphi\le \alpha_\varphi< \delta,$ via
(\ref{rbm}), the formula (\ref{idealQrho}) leads to
\begin{equation}\label{idealQ}
L_{2,\infty}^{(a)}(Q)\leq c\  \frac{2^{dM/t}}{\varphi(2^M)}.
\end{equation}
Using (\ref{idealsdef}) and (\ref{idealQ}), we get
\begin{equation}\label{a_2dMQvar}
a_{2^{dM}}(Q, B_1, B_2)\leq c\  2^{dM(\frac 1t-\frac
12)}{\varphi(2^M)}^{-1}.
\end{equation}
For the case $\frac dt< \delta< \beta_\varphi\le \alpha_\varphi,$ in
the same manner as above, we derive
\begin{equation}\label{a_2dMQdelta}
a_{2^{dM}}(Q, B_1, B_2)\leq c\  2^{dM(\frac 1t-\frac 12-\frac\delta
d)}.
\end{equation}
We take $k=2^{Md}.$ Therefore, under the assumption, $\frac
dt<\beta_\varphi\le\alpha_\varphi <\delta$ \,or \,$\frac dt < \delta
< \beta_\varphi,$ the estimate from above follows from
(\ref{a_2dMPvar}), (\ref{a_2dMPdelta}), (\ref{a_2dMQvar}) and
(\ref{a_2dMQdelta}), by monotonicity of the approximation numbers
and the subadditivity as below,
\begin{equation}
\label{a_kid} a_{2k}({\rm id}, B_1, B_2)\leq a_k(P, B_1, B_2)+
a_k(Q, B_1, B_2).
\end{equation}
{\tt Step 4}.\quad The estimate of $a_k({\rm id}, B_1, B_2)$ from
above in the case of $\mu<\frac dt.$ Inspired by \cite{SV09}, we use
the following division
\begin{equation}
\label{divi3} {\rm
id}=\sum\limits_{m=0}^{M_1}\sum\limits_{j+i=m}P_{j,i}+
\sum\limits_{m=M_1+1}^{M_2}\sum\limits_{j+i=m}P_{j,i}+
\sum\limits_{m=M_2+1}^\infty\sum\limits_{j+i=m}P_{j,i},
\end{equation}
where $M_1,M_2\in\mathbb{N}$\,\,and \,$M_1<M_2$, which will be
determined later on. Using the subadditivity of $s$-numbers, we have
\begin{equation}
\label{dn_divi3} a_{k^\prime}({\rm id}, B_1, B_2)\leq
\triangle_1+\triangle_2+\triangle_3,
\end{equation}
where
\begin{equation*}
\begin{array}{ll}
\triangle_1=\sum\limits_{m=0}^{M_1}
\sum\limits_{j+i=m}a_{k_{j,i}}(P_{j,i}),\quad\quad &
\triangle_2=\sum\limits_{m=M_1+1}^{M_2}
\sum\limits_{j+i=m}a_{k_{j,i}}(P_{j,i}), \\
\triangle_3=\sum\limits_{m=M_2+1}^\infty\sum\limits_{j+i=m}\|P_{j,i}\|,\quad\quad
&k^\prime-1=\sum\limits_{m=0}^{M_2}\sum\limits_{j+i=m}(k_{j,i}-1).
\end{array}
\end{equation*}
Note that for $\triangle_3$, we have $j+i > M_2$ in the sum, and we
take $k_{j,i}=1.$ Now let $k\in\mathbb{N}$ be given. We take
\begin{equation*}
M_1=\left[\frac {\log_2k}d-\frac{\log_2\log_2k}d\right]\quad\quad
{\rm and}\quad \quad M_2=\left[\frac t2\cdot\frac{\log_2k}d\right]
\end{equation*}
where $[a]$ denotes the largest integer smaller than
$a\in\mathbb{R}$ and $\log_2k$ is a dyadic logarithm of $k$. Then
\begin{equation*}
\begin{split}
\triangle_3&=\sum\limits_{m=M_2+1}^\infty\sum\limits_{j+i=m}\|P_{j,i}\|
\leq c_1\sum\limits_{m=M_2+1}^\infty\sum\limits_{j+i=m}\frac
1{2^{j\delta}\varphi(2^i)}
\\ &\leq c_1\sum\limits_{m=M_2+1}^\infty\sum\limits_{j=0}^m\frac
1{2^{j\delta}\varphi(2^{m-j})}
\\ &\leq
\begin{cases}
c_2\sum\limits_{m=M_2+1}^\infty \frac 1{\varphi(2^{m})}\leq c_3
\varphi(2^{M_2})^{-1} \leq c_4 \varphi \big(k^{\frac
t{2d}}\big)^{-1}~&{\rm if}\,~ \delta>\alpha_\varphi,
\\
c_2\sum\limits_{m=M_2+1}^\infty2^{-m\delta}\leq c_3
2^{-M_2\delta}\leq c_4 k^{-\frac{t\delta}{2d}}&{\rm if}\,~
\delta<\beta_\varphi.
\end{cases}
\end{split}
\end{equation*}
Next, we choose proper $k_{j,i}$ for estimating $\triangle_1$ and
$\triangle_2$. If $i+j\leq M_1,$ we take $k_{j,i}=M_{j,i}+1$ such
that $a_{k_{j,i}}(P_{j,i})=0$ and $\triangle_1=0.$ And we obtain
$$\sum\limits_{m=0}^{M_1}\sum\limits_{j+i=m}k_{j,i}\leq
 c_1\sum\limits_{m=0}^{M_1}(m+1)2^{md}\leq c_2M_12^{M_1d}\leq c_3k.$$
Now we give the crucial choice of $k_{j,i}$ for the second sum
$\triangle_2.$  Following Skrzypczak and Vyb\'iral \cite{SV09}, we
take
$$k_{j,i}=[k^{1-\varepsilon}\cdot2^{i\tau_1}\cdot 2^{j\tau_2}],$$
where $\varepsilon,\tau_1, \tau_2$ are positive real numbers such
that
$$\alpha_\varphi+\frac{\tau_1}2<\frac dt,\quad
0<\frac{\tau_1-\tau_2}2<\delta-\alpha_\varphi\quad {\rm and}\quad
\frac{\tau_1 t}{2d}=\varepsilon\quad {\rm if}\,
\delta>\alpha_\varphi,$$
or
$$\ \delta+\frac{\tau_2}2<\frac dt,\quad \  \
0<\frac{\tau_2-\tau_1}2<\beta_\varphi-\delta\quad {\rm and}\quad
\frac{\tau_2 t}{2d}=\varepsilon\quad {\rm if}\,
\delta<\beta_\varphi.$$ Note that the relation, $0<\varepsilon<1,$
holds obviously. Then
$$\sum\limits_{m=M_1+1}^{M_2}
\sum\limits_{j+i=m}k_{j,i}\leq c_1
k^{1-\varepsilon}\sum\limits_{m=M_1+1}^{M_2}2^{m\cdot\max(\tau_1,
\tau_2)} \leq c_2k^{1-\varepsilon}\cdot k^{\frac {t}{2d}\max(\tau_1,
\tau_2)}=c_2 k.$$ And, in terms of (\ref{ar}) and (\ref{idealse2}),
for the case $0<\beta_\varphi\le\alpha_\varphi < \frac dt$ and
$\alpha_\varphi < \delta,$ we have
\begin{equation*}
\begin{split}
\sum\limits_{m=M_1+1}^{M_2}\sum\limits_{j+i=m}a_{k_{j,i}}(P_{j,i})
&\leq c_1\sum\limits_{m=M_1+1}^{M_2}\sum\limits_{j+i=m} \frac
1{2^{j\delta}\varphi(2^i)}
2^{(i+j)d/t}[k^{1-\varepsilon}\cdot2^{i\tau_1}\cdot
2^{j\tau_2}]^{-\frac 12}
\\
&\leq c_2 k^{-\frac 12(1-\varepsilon)}\sum\limits_{m=M_1+1}^{M_2}
2^{md/t}\sum\limits_{j=0}^m \frac{ 2^{-j(\delta+\frac{\tau_2}2)}}{
2^{(m-j)\tau_1/2} \cdot \varphi(2^{m-j}) }
\\
&\leq c_3 k^{-\frac 12(1-\varepsilon)}\sum\limits_{m=M_1+1}^{M_2}
2^{m(d/t-\tau_1/2)}\frac
1{\varphi(2^{m})}\\
&\leq c_4 \varphi(2^{M_2})^{-1} \leq c_5 \varphi \big(k^{\frac
t{2d}}\big)^{-1}.
\end{split}
\end{equation*}
Similarly, for the case $0<\delta < \frac dt$ and $\delta<
\beta_\varphi,$ we derive
\begin{equation*}
\begin{split}
\sum\limits_{m=M_1+1}^{M_2}\sum\limits_{j+i=m}a_{k_{j,i}}(P_{j,i})
&\leq c_1 k^{-\frac 12(1-\varepsilon)}\sum\limits_{m=M_1+1}^{M_2}
2^{md/t}\sum\limits_{j=0}^m\frac{ 2^{-j(\delta+\frac{\tau_2}2)}}{
2^{(m-j)\tau_1/2} \cdot \varphi(2^{m-j}) }\\
&\leq c_2 k^{-\frac 12(1-\varepsilon)}\sum\limits_{m=M_1+1}^{M_2}
2^{m(d/t-\tau_2/2)}2^{-m\delta}\\
&\leq c_3 2^{-M_2\delta} \leq c_3 k^{-\frac{t\delta}{2d}}.
\end{split}
\end{equation*}
Hence the estimate from above is now finished as required.
\\{\tt Step 5}.\quad The estimate of $a_n({\rm id}, B_1, B_2)$ from
below. Consider the following diagram
\begin{equation}
\begin{CD}
\ell_{p_1}^{M_{j,i}} @>S_{j,i}>>
\ell_{q_1}(2^{j\delta}\ell_{p_1}(w))
\\
@VV{\rm id_1}V @VV{\rm id}V\\
\ell_{p_2}^{M_{j,i}} @<T_{j,i}<< \ell_{q_2}(\ell_{p_2})
\end{CD}
\end{equation}
Here,
\begin{equation*}
\begin{array}{rl}
(S_{j,i}\eta)_{u,v}:=&\left\{
\begin{array}{ll}\eta_{\varphi(u,v)}\,
&{\rm if}\quad (u,v)\in I_{j,i}, \\ 0 \, &{\rm otherwise,}
\end{array}\right.\\
(T_{j,i}\lambda)_{\varphi(u,v)}:=&\lambda_{u,v},\quad\quad\quad
(u,v)\in I_{j,i},\end{array}
\end{equation*}
and $\varphi$ denotes a bijection of $I_{j,i}$ onto $\{1, \ldots,
M_{j,i}\},\, j\in\mathbb{N}_0, i\in\mathbb{N}_0.$ Then we observe
\begin{equation*}
\begin{array}{ll}
S_{j,i}\in \mathcal {L}\left(\ell_{p_1}^{M_{j,i}},\,
\ell_{q_1}(2^{j\delta}\ell_{p_1}(w))\right) \quad  &{\rm
and}\quad \|S_{j,i}\|=2^{j\delta}\varphi(2^i),\\
T_{j,i}\in  \mathcal {L}\left(\ell_{q_2}(\ell_{p_2}),\,
\ell_{p_2}^{M_{j,i}}\right) \quad  &{\rm and}\quad \|T_{j,i}\|=1.
\end{array}
\end{equation*}
Hence we obtain
\begin{equation}\label{diagST}
a_k(\,{\rm id}_1)\leq\|S_{j,i}\|~\|T_{j,i}\|\,a_k(\,{\rm id}).
\end{equation}

$(i)$\, Let $\frac dt<\delta<\beta_\varphi$. We denote
$N:=M_{j,0}=|I_{j,0}|\sim 2^{jd},\,\,j\ge\frac 2d.$ Then
$$\|S_{j,0}\|\leq C 2^{j\delta}\quad{\rm and}\quad\|T_{j,0}\|=1.$$
Put $m=\frac N4\sim2^{jd-2}.$ And for sufficiently large $N$ we have
$m\ge N^{2/t}$ since $t>2.$ Consequently,
\begin{equation*}
a_m({\rm id}_1,\,\ell_{p_1}^N,\,\ell_{p_2}^N)\sim
N^{1/t}m^{-1/2}\sim 2^{(jd-2)(\frac 1t-\frac 12)}\sim2^{(jd-2)(\frac
1t-\frac 12)}.
\end{equation*}
Using (\ref{diagST}), we obtain
\begin{equation*}
a_{2^{jd-2}}({\rm id})\ge C_12^{-j\delta}2^{(jd-2)(\frac 1t-\frac
12)}\ge C_22^{(jd-2)(\frac 1t-\frac 12-\frac\delta d)},
\end{equation*}
and, by monotonicity of the approximation numbers, for any
$k\in{\mathbb N},$
\begin{equation}
a_k({\rm id})\ge
 C_3 k^{-(\frac\delta d+\frac 12-\frac 1t)}.
\end{equation}

$(ii)$\, Let $\frac dt<\beta_\varphi\leq\alpha_\varphi<\delta$. We
consider $N:=M_{0,i}=|I_{0,i}|\sim2^{id},\,\,i\ge\frac 2d.$ Then
$$\|S_{0,i}\|\leq C\varphi(2^i)\quad{\rm
and}\quad\|T_{0,i}\|=1.$$Also put $m=\frac N4\sim2^{id-2}.$ Hence we
have similarly, for any $k\in{\mathbb N},$
\begin{equation}
a_k({\rm id})\ge C_1n^{-(\frac 12-\frac
1t)}\varphi\big(k^{1/d}\big)^{-1}.
\end{equation}

$(iii)$\, Let $\delta\leq\frac dt~{\rm and}~\delta<\beta_\varphi.$
We select the same $N,\,S,\,{\rm and}~\,T$ as in point (i) and take
$m=\big[N^{2/t}\big]\leq\frac N4$ for sufficiently large $N.$ Then
$N^{1/t}m^{-1/2}\sim 1.$ Hence by Lemma \ref{an1} and (\ref{diagST})
we obtain
\begin{equation*}
a_m({\rm id})\ge C2^{-j\delta}=C2^{-jd\frac 2t\frac{t\delta}{2d}},
\end{equation*}
and then, for any $k\in{\mathbb N},$
\begin{equation}
a_k({\rm id})\ge C_1k^{-\frac{t\delta}{2d}}.
\end{equation}

$(iv)$\, Let $0<\beta_\varphi\leq\alpha_\varphi<\frac dt~{\rm
and}~\alpha_\varphi<\delta.$ We select the same $N,\,S,\,{\rm
and}~\,T$ as in point (ii) and take $m=\big[N^{ 2/t}\big]$ in the
same way as in point (iii). Then similarly
\begin{equation*}
a_m({\rm id})\ge C\varphi\big(2^i\big)^{-1}=C\varphi\big(2^{id\frac
2t \frac t{2d}}\big)^{-1},
\end{equation*}
and in consequence, for any $k\in{\mathbb N},$
\begin{equation}
a_k({\rm id})\ge C_1 \varphi\big(k^{\frac t{2d}}\big)^{-1}.
\end{equation}
The proof of the proposition is now complete.
\end{proof}

\begin{prop}\label{an5}
Suppose $0<p_1\le 1$\,and $p_2=\infty$. Let $w\in \mathcal{W}_1
\cap\mathcal{W}_2$, and let $\varphi \in \mathcal{V}$ be an
associated function in the sense of Definition \ref{w2}. Then
\begin{equation}
a_{k}\Big({\rm id}, \ell_{q_1}(2^{j\delta}\ell_{p_1}(w)),
\ell_{q_2}(\ell_{p_2})\Big) \sim
\begin{cases}
 k^{- \frac
12}\varphi\big(k^{\frac 1d}\big)^{-1} & {\rm if}\,~
0<\beta_\varphi\le\alpha_\varphi <\delta,
\\
k^{-\frac\delta d - \frac 12} & {\rm if}\,~ 0 < \delta <
\beta_\varphi.
\end{cases}
\end{equation}
\end{prop}
\begin{proof}
We only sketch the proof since we can proceed as in the last proof.

\emph{Step 1} (Upper Estimates).  $0<p_1\le1$ and $p_2=\infty$ imply
$t=\infty$. We select $0<\lambda<1$ such that
$\frac\lambda{2(1-\lambda)}<\frac\mu d$, with
$\mu=\min(\beta_\varphi,\delta)$. The inequality $\lambda\cdot \frac
1s\le\frac 1s-\frac 12$ holds if and only if $\frac 1s\ge\frac
1{2(1-\lambda)}$, where $0<\lambda<1$. Then, we find by
(\ref{an1infupp}) that for any $N\in\mathbb{N}$
\begin{equation}\label{idealseh}
L_{h,\infty}^{(a)}({\rm id}, \ell_{p_1}^N, \ell_{p_2}^N)\leq C\
N^{\frac\lambda{2(1-\lambda)}},\ \ \ {\rm if}\ \ \frac 1h=\frac
1{2(1-\lambda)},
\end{equation}
\begin{equation}\label{idealsh}
L_{s,\infty}^{(a)}({\rm id}, \ell_{p_1}^N, \ell_{p_2}^N)\leq C\
N^{(\frac 1s-\frac 12)},\ \ \ {\rm if}\ \ \frac 1s>\frac
1{2(1-\lambda)}.
\end{equation}
As to the precise definitions of $P$ and $Q$, we refer to the
counterpart of the last proof again. For the estimation of $a_n(P)$,
we choose $s$ such that $\frac 1s>\frac 1{2(1-\lambda)}$\ and\
$d(\frac 1s-\frac 12)>\mu$, and proceed by using (\ref{idealsh}).
For the estimation of $a_n(Q)$, we choose $s=h=2(1-\lambda)$, and
use (\ref{idealseh}) instead. Note that $\lambda\cdot \frac 1h=\frac
1h-\frac 12<\frac\mu d$.

\emph{Step 2} (Lower Estimates). We only need to consider two cases,
$0<\delta<\beta_\varphi$\
 or $0<\beta_\varphi\le\alpha_\varphi<\delta$. And we choose $\ell=\big[\frac N2\big]$ where $N$ is
taken in the same way as in point (a) or (b) of Step 5 of
Proposition \ref{an4}, respectively, and use (\ref{an1inflow}).
\end{proof}

\subsection{Proofs of main assertions}\label{akg}
First, let us point out that the last two results, Propositions
\ref{an4} and \ref{an5}, are independent of the fine indices\ $q_1$
and\ $q_2$ of the Besov spaces. Afterwards, in the proof for the
Banach space case ($1\le p_1,\,p_2,\,q_1,\,q_2\le \infty$), the
restrictions $q_1,q_2\ge1$ could be lifted. So in the following
proofs, we can restrict the
attention to\ $p_1$ and\ $p_2$. \vspace{0.4cm} \\
\emph{Proof of Theorem \ref{an}.}\ \ Based on Proposition
~\ref{Besov_des}, we transfer Proposition \ref{an4} for weighted
sequence spaces to weighted Besov spaces, and extend the result
trivially to the quasi-Banach space case ($p_1<1$ and
$2<p_2<\infty$) in terms of point (iii) of Lemma \ref{an1}. Then the
transformation of Proposition \ref{an5} finishes the proof of point
(iii) in Theorem \ref{an}. Afterwards, the other two points, (i) and
(ii), follow from the proof of Proposition 13 and Proposition 15
respectively in \cite{Sk05}, in view of Lemma \ref{an1} and Lemma
\ref{an2}, respectively. \qed
\\
\quad \\
\emph{Proof of Theorem \ref{kn}.}\ \ The estimate in the Banach
space case ($p_1, p_2\ge 1$) follows literally from the counterpart
of Theorem \ref{an}, with a few replacements as adopted in the
proofs of Propositions 3.7 and 3.8 in \cite{ZF10}. For the
quasi-Banach space case ($p_1<1$\ or\ $p_2< 1$), though there are
many differences between both widths of the Euclidean ball, we can
prove it in a similar way as in the proof of Theorem 2.5 in
\cite{ZF10} (see also \cite{ZF11}), in terms of Lemma \ref{rab}.
Please note that the same formula from Lemma \ref{an2} is not true
for Kolmogorov numbers if $p_2< 1$.\qed
\\
\quad \\
\emph{Proof of Theorem \ref{gn}.}\ \ The estimate in the Banach
space case ($p_1, p_2\ge 1$) follows trivially from Theorem \ref{kn}
due to the duality between Kolmogorov and Gelfand numbers of the
Euclidean ball, cf. (\ref{dualc*d}) and (\ref{duald*c}). For the
quasi-Banach space case ($p_1<1$\ or\ $p_2< 1$), there are still
many differences between both widths of the Euclidean ball. And we
can finish it exactly as the proof of Theorem 2.8 in \cite{ZF10}
(see also \cite{ZF11}), in terms of Lemma \ref{rab}. Of course, the
same formula from Lemma \ref{an2} is also valid for Gelfand numbers
if $p_1<1$\ or\ $p_2< 1$, consult literally \cite{Pie78}, Section
11.11.4. \qed

\section{Widths for weights of typical types}
In this section we recall polynomial weights and mainly investigate
polynomial weights with small perturbation. For details like basic
properties and examples for both types, one can consult
\cite{KLSS06} and references therein.
\subsection{Polynomial weights}

Let $\alpha>0.$ We put
\begin{equation}\label{w_a}
w_\alpha(x):=(1+|x|^2)^{\alpha/2},\quad x\in \mathbb{R}^d.
\end{equation}
Then $w_\alpha\in \mathcal{W}_1 \cap\mathcal{W}_2$. The associated
function $\varphi$ is denoted by $\varphi(t)=|t|^\alpha,$ and we
have $\alpha_\varphi=\beta_\varphi=\alpha.$

Now we recall the necessary and sufficient condition for compactness
of the embeddings under consideration, which was proved in
\cite{HT94}, cf. also \cite{ET96, KLSS06}.
\begin{lemma}
Let $w_\alpha$ be as in (\ref{w_a}) for some $\alpha>0.$ The
embedding
\begin{equation}\label{BBa}
B_{p_1,q_1}^{s_1}(\mathbb{R}^d, w_\alpha)\hookrightarrow
B_{p_2,q_2}^{s_2}(\mathbb{R}^d)
\end{equation}
is compact if and only if $ \min(\alpha, \delta)>  d/{p^*}.$
\end{lemma}

Applying Theorems \ref{an}, \ref{kn} and \ref{gn} leads to the
corresponding results on approximation numbers, Gelfand and
Kolmogorov numbers
 in weighted Besov spaces with polynomial weights, respectively,
which were already proved in \cite{Sk05,ZF10,ZF11}.

\subsection{Small perturbations of polynomial weights}

Let $\psi : [0,\infty) \rightarrow (0,\infty)$ be a positive and
continuous function such that
$$\psi(t)=\exp\Big\{\int_1^t\varepsilon(u)\frac{du}{u}\Big\},\quad t\in [1,\infty),$$
for some bounded and measurable function $\varepsilon$ satisfying
$\lim_{u\rightarrow}\varepsilon(u)=0.$ Then we say $\psi$ is a
(normalized) slowly varying function. And for $\alpha>0$ the
function
\begin{equation}\label{w_ap}
w_{\alpha,\psi}(x):=(1+|x|^2)^{\alpha/2}\psi(|x|),\quad x\in
\mathbb{R}^d,
\end{equation}
belongs to $\mathcal W_2$. The associated function $\varphi$ may be
represented as $\varphi(t)=(1+t^2)^{\alpha/2}\psi(t),$ and we have
$\alpha_\varphi=\beta_\varphi=\alpha$. The following proposition is
taken from \cite{KLSS06}.
\begin{prop}
Let $w_{\alpha,\psi}$ be as in (\ref{w_ap}) for some $\alpha>0.$ The
embedding
\begin{equation}\label{BBap}
B_{p_1,q_1}^{s_1}(\mathbb{R}^d, w_{\alpha,\psi})\hookrightarrow
B_{p_2,q_2}^{s_2}(\mathbb{R}^d)
\end{equation}
 is compact if and only if one of the
following conditions is satisfied:\vspace{-0.2cm}
\begin{enumerate}
\item[$(i)$]\ $ \min(\alpha, \delta)>  d/{p^*};$
\vspace{-0.2cm}

\item[$(ii)$]\ $\delta>\alpha = d / {p^*}$ ~ and ~ $\int_1^\infty
\psi(t)^{-p^*}\frac{dt}t<\infty.$
\end{enumerate}
\end{prop}

This time in terms of Theorem \ref{an}, Theorem \ref{kn} and Theorem
\ref{gn}, we obtain the following assertions. We wish to mention
that the proofs  of the following three theorems for the special
case, $0<\alpha=\frac dp\ (<\delta)$, will be given in a sketch
later on.

\begin{theorem}\label{anap}
Let $\delta> \alpha >0$ and $\frac 1{\tilde{p}}=\frac\alpha d+\frac
1{p_1}.$ Denote by $a_k$ the $k$th approximation number of the
embedding (\ref{BBap}). \vspace{-0.2cm}
\begin{enumerate}
\item[$(i)$] Suppose in addition that $0< p_1 \le p_2 \le
2$\,\,or\,\,$2\le p_1\le p_2\le \infty.$ Then
$$a_k \sim k^{-\frac \alpha d}\psi\big(k^{\frac 1d}\big)^{-1}.$$
\vspace{-0.8cm}

\item[$(ii)$] Suppose that in addition to the general hypotheses,
$\tilde{p}< p_2 < p_1 \le \infty.$ Then
$$
a_k \sim
\begin{cases}
 k^{-\frac \alpha d - (\frac 1{p_1}-\frac
1{p_2})}\psi\big(k^{\frac 1d}\big)^{-1} ~ &{\rm if} ~ \alpha>\frac
dp,
\\
\big(\int_{k^{1/d}}^\infty \psi(t)^{-p}\frac{dt}t\big)^{1/p}
 ~ &{\rm if} ~\alpha =\frac dp.
\end{cases}
$$
\vspace{-0.6cm}

\item[$(iii)$] Suppose that in addition to the general hypotheses, $0< p_1
< 2 < p_2\le \infty$.  Then
$$
a_k \sim
\begin{cases}
 k^{-\frac \alpha d - (\frac 1{2}-\frac
1{t})}\psi\big(k^{\frac 1d}\big)^{-1} ~ &{\rm if} ~ \alpha>\frac dt,
\\
 k^{-\frac{\alpha t}{2d}}\psi\big(k^{\frac
t{2d}}\big)^{-1} ~ &{\rm if} ~ \alpha < \frac dt.
\end{cases}
$$where
$t=\min(p_1^\prime,p_2)$. \vspace{-0.2cm}
\end{enumerate}

\end{theorem}

\begin{theorem}\label{knap}
Let $\delta> \alpha >0,\,\theta =
\frac{1/{p_1}-1/{p_2}}{1/2-1/{p_2}}$ and $\frac
1{\tilde{p}}=\frac\alpha d+\frac 1{p_1}.$ Denote by $d_k$ the $k$th
approximation number of the embedding (\ref{BBap}). \vspace{-0.2cm}
\begin{enumerate}
\item[$(i)$] Suppose in addition that $0< p_1 \le p_2 \le
2$\,\,or\,\,$2< p_1= p_2\le \infty.$ Then \vspace{-0.2cm}
$$d_k \sim k^{-\frac \alpha d}\psi\big(k^{\frac 1d}\big)^{-1}.$$
\vspace{-0.8cm}

\item[$(ii)$] Suppose that in addition to the general hypotheses,
$\tilde{p} < p_2 < p_1 \le \infty.$ Then \vspace{-0.2cm}
$$
d_k \sim
\begin{cases}
k^{-\frac \alpha d - (\frac 1{p_1}-\frac 1{p_2})}\psi\big(k^{\frac
1d}\big)^{-1} ~ &{\rm if} ~ \alpha>\frac dp,
\\
\big(\int_{k^{1/d}}^\infty \psi(t)^{-p}\frac{dt}t\big)^{1/p}
 ~ &{\rm if}~ \alpha=\frac dp.
\end{cases}
$$
\vspace{-0.6cm}

\item[$(iii)$] Suppose that in addition to the general hypotheses, $0< p_1
< 2 < p_2\le \infty.$ Then \vspace{-0.2cm}
$$d_k \sim
\begin{cases}
k^{-\frac\alpha d+\frac 1{p_2} - \frac 12}\psi\big(k^{\frac
1d}\big)^{-1} ~ &{\rm if} ~ \alpha>\frac d{p_2},
\\
 k^{-\frac {p_2\alpha}{2d}}\psi\big(k^{\frac
{p_2}{2d}}\big)^{-1} ~ &{\rm if} ~ \alpha < \frac d{p_2}.
\end{cases}
$$
\vspace{-0.6cm}

\item[$(iv)$] Suppose that in addition to the general hypotheses, $2\le p_1
< p_2\le \infty.$ Then \vspace{-0.2cm}
$$d_k \sim
\begin{cases}
k^{-\frac\alpha d+\frac 1{p_2} - \frac 1{p_1}}\psi\big(k^{\frac
1d}\big)^{-1} ~ &{\rm if} ~ \alpha>\frac d{p_2}\theta,
\\
 k^{-\frac {p_2\alpha}{2d}}\psi\big(k^{\frac
{p_2}{2d}}\big)^{-1} ~ &{\rm if} ~ \alpha < \frac d{p_2}\theta.
\end{cases}
$$
\end{enumerate}

\end{theorem}

\begin{theorem}\label{gnap}
Let $\delta> \alpha >0,\ \theta_1 =
\frac{1/{p_1}-1/{p_2}}{1/{p_1}-1/2}$ and $\frac
1{\tilde{p}}=\frac\alpha d+\frac 1{p_1}.$ Denote by $c_k$ the $k$th
approximation number of the embedding (\ref{BBap}).
 \vspace{-0.2cm}
\begin{enumerate}
\item[$(i)$] Suppose in addition that $2\le p_1\le p_2\le
\infty$\,\,or\,\,$0< p_1 = p_2< 2.$ Then
$$c_k \sim k^{-\frac \alpha d}\psi\big(k^{\frac 1d}\big)^{-1}.$$
\vspace{-0.8cm}

\item[$(ii)$] Suppose that in addition to the general hypotheses,
$\tilde{p} < p_2 < p_1 \le \infty.$ Then
$$
c_k \sim
\begin{cases}
 k^{-\frac \alpha d - (\frac 1{p_1}-\frac
1{p_2})}\psi\big(k^{\frac 1d}\big)^{-1} ~ &{\rm if} ~ \alpha>\frac
dp,
\\
 \big(\int_{k^{1/d}}^\infty
\psi(t)^{-p}\frac{dt}t\big)^{1/p}
 ~ &{\rm if} ~ \alpha=\frac dp.
\end{cases}
 $$
\vspace{-0.6cm}

\item[$(iii)$] Suppose that in addition to the general hypotheses, $0< p_1
< 2 < p_2\le \infty.$ Then
$$
c_k \sim
\begin{cases}
 k^{-\frac\alpha d+\frac 12 - \frac
1{p_1}}\psi\big(k^{\frac 1d}\big)^{-1} ~ &{\rm if} ~ \alpha>\frac
d{p_1^\prime},
\\
 k^{-\frac {p_1^\prime\alpha}{2d}}\psi\big(k^{\frac
{p_1^\prime}{2d}}\big)^{-1} ~ &{\rm if} ~ \alpha < \frac
d{p_1^\prime}.
\end{cases}
$$
\vspace{-0.6cm}

\item[$(iv)$] Suppose that in addition to the general hypotheses, $0< p_1 <
p_2 \le 2.$ Then
$$
c_k \sim
\begin{cases}
 k^{-\frac\alpha d+\frac 1{p_2} - \frac
1{p_1}}\psi\big(k^{\frac 1d}\big)^{-1} ~ &{\rm if} ~ \alpha>\frac
d{p_1^\prime}\theta_1,
\\
 k^{-\frac {p_1^\prime\alpha}{2d}}\psi\big(k^{\frac
{p_1^\prime}{2d}}\big)^{-1} ~ &{\rm if} ~ \alpha < \frac
d{p_1^\prime}\theta_1.
\end{cases}
$$
\end{enumerate} \vspace{-0.2cm}
\end{theorem}

\begin{re}
In any point of these above three theorems with the opposite
assumption $\alpha> \delta$, the asmptotic behaviour of the
approximation, Gelfand and Kolomogrov numbers may be directly
determined in terms of our main Theorems for the general case.
\end{re}
The proof of the above three theorems follows trivially in most
cases. We only need to prove the case $\alpha=d/p$ in point (ii) of
them, respectively.

Before passing to the proofs we collect some known results which
will be needed later. The first result, on approximation and Gelfand
numbers of diagonal operators, can be found in \cite{Pie78}, Theorem
11.11.4, where only the Banach space case $1\le p_2 < p_1\le \infty$
was considered. However, the same proof works also in the
quasi-Banach case, i.e. when $0 < p_1 < 1$ or $0 < p_2 < 1$.

\begin{lemma}\label{diagoper_ag}
Let $0 < p_2 < p_1\le\infty,\ 1/p = 1/{p_2}- 1/{p_1},$ and let
$(\sigma_n)\in\ell_p$ be a non-increasing sequence of non-negative
real numbers. Let $D_\sigma$ be a diagonal operator from
$\ell_{p_1}$ to $\ell_{p_2}$, defined by $(x_n)\mapsto
(\sigma_nx_n)$. Then
\begin{equation}\label{diag_ag}
a_k(D_\sigma:\ell_{p_1}\rightarrow\ell_{p_2})=c_k(D_\sigma:\ell_{p_1}\rightarrow\ell_{p_2})=
\Big(\sum_{n=k}^\infty\sigma_n^p\Big)^{1/p}.
\end{equation}
\end{lemma}

Let us mention that, in contrast to the above lemma, the same
formula is not true for Kolomogrov numbers. Inspired by the proof of
Theorem 1 in \cite{Ku08}, we tried by elementary methods to prove
the desired formula with ``$=$'' replaced by ``$\sim$'', as shown in
(\ref{diag_kn}), but failed. Alternatively, we establish it in terms
of an inequality between entropy and Kolomogrov numbers. The entropy
numbers is defined as follows: if $T\in\mathcal{L}(X,Y)$, then we
set
$$e_k(T)=\inf\{\varepsilon>0\ |\ \exists y_1, \ldots,
, y_{2^{k-1}}\in Y, \forall x\in X,\|x\|\le 1,\exists i\le
2^{k-1},\|Tx-y_i\|\le\varepsilon\}.$$ We refer to \cite{ET96,Pie78}
for detailed discussions of this concept and further references.

The next result is essentially due to Carl \cite{Car81,CS90} who
proved the relation between entropy and approximation numbers in the
Banach space case, for the extension to the quasi-Banach case see
Triebel \cite{ET96}.

\begin{lemma}\label{relation_ed}
Let $T \in \mathcal {L}(X, Y)$. Assume that $f:\
\mathbb{N}\rightarrow \mathbb{R}$ is a positive increasing function
with $f(k)\sim f(2k).$ Then there is a constant $C>0$ depending only
on $f$ and the quasi-triangle constant of $Y$ such that for all
$n\in\mathbb{N}$ we have
\begin{equation}\label{ed}
\sup\limits_{k\le n}f(k)e_k(T)\leq C\sup\limits_{k\le n}f(k)d_k(T).
\end{equation}
\end{lemma}

The proof of the above lemma mimics that of Proposition 5.1 in
Pisier \cite{Pis89} with minor changes and follows by Theorem 1.3.3
in \cite{ET96}. Note that we identify $X$ as a quotient of
$\ell_p(I),\ (0<p\le 1)$, for some $I$, and apply the metric lifting
property of $\ell_p(I)$ in the class of $p$-normed spaces, see
Proposition C.3.6 in Pietsch \cite{Pie78}.

Again we obtain the estimate of Kolomogrov numbers of diagonal
operators.
\begin{prop}\label{diagoper_kn}
Let $p_1, p_2, p,(\sigma_n)$ and $D_\sigma$ be as in Lemma
\ref{diagoper_ag}. Assume that the tail sequence
$h_k=(\sum_{n=k}^\infty\sigma_n^p)^{1/p}$ satisfies the doubling
condition $h_k\sim h_{2k}$. Then
\begin{equation}\label{diag_kn}
d_k(D_\sigma:\ell_{p_1}\rightarrow\ell_{p_2})\sim h_k.
\end{equation}
\end{prop}

The proof of this proposition follows from (\ref{acd}), Lemma
\ref{relation_ed}, Theorem 1 in \cite{Ku08} and Theorem 3.3 in
\cite{CK09}.

Now we are ready to finish the proofs of Theorems \ref{anap},
\ref{knap} and \ref{gnap}. Their proofs for the case $\alpha=d/p>0$
in point (ii), respectively, follow literally the proof of Theorem 7
in \cite{Ku08} by virtue of Lemma \ref{diagoper_ag} or Proposition
\ref{diagoper_kn}, where a rearrangement technique is crucial.

\begin{re}
It is remarkable that all of our results in this paper do not depend
on the microscopic parameters\ $q_1$ and\ $q_2$ of the Besov spaces.
Therefore, all of them remain valid if we replace Besov spaces by
Triebel-Lizorkin spaces $F_{p,q}^s(\mathbb{R}^d, w)$.
\end{re}

\begin{re}
Finally, we wish to mention the following open question. In the
limiting case of small perturbations of polynomial weights,
$\alpha=\delta,$ how do the related widths behave? Different from
the non-limiting case, the corresponding behavior herein maybe
depend on the parameters\ $q_1$ and\ $q_2$ to a certain extent. Some
ideas from \cite{Ha97,KLSS05,MO88} might be helpful.
\end{re}

\section*{Acknowledgments}
~~~ The authors are extremely grateful to Ant\'onio M. Caetano,
Fanglun Huang, Thomas K$\ddot{\rm u}$hn, Erich Novak and Leszek
Skrzypczak for their direction and help on this work. This research
was supported by the Natural Science Foundation of China (Grant No.
10671019) and Anhui Provincial Natural Science Foundation (Grant No.
090416230).

\end{document}